\documentclass{amsproc}
\usepackage{centernot}
\usepackage{mathtools}
\usepackage{stmaryrd}
\usepackage{tikz-cd}
\usepackage{enumitem}

\makeatletter
\newcommand{\xMapsto}[2][]{\ext@arrow 0599{\Mapstofill@}{#1}{#2}}
\def\Mapstofill@{\arrowfill@{\Mapstochar\Relbar}\Relbar\Rightarrow}
\makeatother

\usepackage[mathcal]{euscript}
\usepackage{amssymb, amsfonts, amsmath}
\usepackage{xypic}
\usepackage{amscd}
\usepackage{dutchcal}
\usepackage{amsthm}
\usepackage[margin=1.5in]{geometry}
\usepackage{xcolor}
\usepackage{comment}

\usepackage[style = alphabetic, doi = false, url=false, isbn = false]{biblatex}
\usepackage{hyperref}
\usepackage{xurl}
\hypersetup{breaklinks=true}
\renewbibmacro{in:}{}
\definecolor{winered}{rgb}{0.5,0,0}
 \usepackage{hyperref}
 \hypersetup{citecolor = winered, linkcolor = winered, menucolor = winered, colorlinks = true}

\newtheorem{theorem}{Theorem}

\newtheorem{theorem-definition}[theorem]{Theorem-Definition}
\newtheorem{proposition}[theorem]{Proposition}

\theoremstyle{definition}
\newtheorem{definition}[theorem]{Definition}

\newtheorem{remark}[theorem]{Remark}
\newtheorem*{theorem*}{Theorem}

\numberwithin{equation}{section} \numberwithin{figure}{section}

\numberwithin{equation}{section}

\newcommand{\E}{\E_{\infty}}

\usepackage{scalerel,amssymb}

\def\msquare{\mathord{\scalerel*{\Box}{gX}}}
\newtheorem*{ack}{Acknowledgments}

\usepackage{graphicx}
\begin{document}

\title[]{An algebraic model for the free loop space}
\author{Manuel Rivera}
\address{Manuel Rivera, Department of Mathematics, Purdue University, 150 N. University Street, West Lafayette, IN 47907-2067}
\email{\href{mailto:manuelr@purdue.edu}{manuelr@purdue.edu}}
\maketitle

\begin{abstract} We describe an algebraic chain level construction that models the passage from an arbitrary topological space to its free loop space. The input of the construction is a categorical coalgebra, i.e. a curved coalgebra satisfying certain properties, and the output is a chain complex. The construction is a modified version of the coHochschild complex of a differential graded (dg) coalgebra. When applied to the chains on an arbitrary simplicial set $X$, appropriately interpreted, it yields a chain complex that is naturally quasi-isomorphic to the singular chains on the free loop space of the geometric realization of $X$. We relate this construction to a twisted tensor product model for the free loop space constructed using the adjoint action of a dg Hopf algebra model for the based loop space. 

 \end{abstract} 
 
\section{Introduction}

To any space $Y$ one can naturally associate a new space $LY$ defined as the set of continuous maps from the circle $S^1$ to $Y$ equipped with the compact-open topology. The space $LY$ is called the \textit{free loop space} of $Y$ and may be equipped with a natural $S^1$-action given by rotation of loops. In this article, we present an algebraic chain level model for the passage $Y \mapsto LY$ that does not assume any hypothesis on the underlying space $Y$. The construction requires a small amount of data and, consequently, is potentially useful for calculations and for studying how the algebraic topology of a geometric space manifests at the level of the free loop space. 

The main idea is inspired by the following picture. Suppose $X$ is a simplicial set. For the purposes of capturing the intuitive idea, the reader may assume $X$ arises from a simplicial complex equipped with a total ordering of its vertices. We consider ordered sequences $(\sigma_0, \cdots, \sigma_p)$ of simplicies in $X$ such that the last vertex of $\sigma_i$ is the first vertex of $\sigma_{i+1}$ for $i=0, \cdots, p-1$ and the last vertex of $\sigma_p$ is the first vertex of $\sigma_0$. These ordered sequences of simplices were called \textit{closed necklaces} in \cite{rivera-saneblidze}. 

We want to think of a closed necklace in $X$ as a family of free loops in $|X|$, the geometric realization of $X$, parameterized by a cube of an appropriate dimension. More precisely, adapting a classical construction of Adams described in \cite{adams}, it is possible to decompose a closed necklace $(\sigma_0, \cdots, \sigma_p)$ into a family of free loops in $|X|$ parameterized by a particular subdivision of a cube of dimension $|\sigma_0|+ \cdots +|\sigma_p|-p$ having the following properties:
\\
$(i)$ the base points of every loop in this family always lie inside the special simplex $\sigma_0$, and
\\
$(ii)$ the boundary of such family may be described in terms of all ``sub-closed necklaces" of codimension $1$.

In \cite{rivera-saneblidze}, this idea was used to construct a \textit{combinatorial} model for $L|X|$ given by gluing a set of polyhedra indexed by closed necklaces in $X$.  In the present paper, we are concerned with an \textit{algebraic} version of this construction. Namely, for an arbitrary commutative ring $\mathbf{k}$, we describe a functorial construction that produces a $\mathbf{k}$-chain complex, directly from the natural algebraic structure of the normalized $\mathbf{k}$-chains $C_*(X)$ suitable interpreted, that computes the $\mathbf{k}$-homology of $L|X|$. We highlight four essential observations that are used in this algebraic construction. 
\\
\begin{enumerate} [wide, labelwidth=!, labelindent=0pt]
\item   The first observation is that the graded $\mathbf{k}$-module freely generated by closed necklaces in $X$ may be described algebraically in terms of the \textit{cotensor product} for a bicomodule structure of the normalized chains $C_*(X)$ over the coalgebra $C_0(X)=\mathbf{k}[X_0]$ generated by the set of vertices of $X$. This bicomodule structure on $C_*(X)$ is determined by projecting a simplex to its first or last vertex. Then for any pair of simplices $\sigma_0$ and $\sigma_1$ in $X$, requiring the last vertex of $\sigma_0$ to be the first vertex of $\sigma_1$ is equivalent to requiring the tensor $\sigma_0 \otimes \sigma_1 \in C_*(X) \otimes_{\mathbf{k}} C_*(X)$ to lie inside the sub-$\mathbf{k}$-module \[C_*(X) \square_{C_0(X)} C_*(X) \subseteq C_*(X) \otimes_{\mathbf{k}} C_*(X),\] where $\square_{C_0(X)}$ denotes the cotensor product of $C_0(X)$-bicomodules. 
\\

\item The second observation is that the boundary of a closed necklace representing a family (or ``chain") of free loops may be described algebraically in terms of certain simplicial face maps and the Alexander-Whitney coproduct \[\Delta \colon C_*(X) \to C_*(X) \square_{C_0(X)} C_*(X).\] The resulting description of the boundary is reminiscent of the differential of the \textit{coHochschild complex} of a differential graded (dg) coalgebra as studied in \cite{Doi}, \cite{hesscohochschild}, and other articles. However, now we are in the different context of comonoids in the category of bicomodules over a coalgebra with cotensor product as monoidal structure. 
\\

\item The third observation is that the relevant structure of $C_*(X)$ for our purposes may be packaged as a curved coalgebra satisfying certain properties, which we call a \textit{categorical coalgebra}. This notion is inspired by Holstein and Lazarev's categorical Koszul duality theory \cite{holstein-lazarev}. Any categorical coalgebra $C$ gives rise to a comonoid in the monoidal category of $C_0$-bicomodules with cotensor product. Furthermore, any categorical coalgebra gives rise to a dg category through a many object version of the cobar construction. The main construction of the article is then a version of the coHochschild chain complex for categorical coalgebras that coincides with the classical coHochschild complex when restricted to connected dg coalgebras. It has a categorical coalgebra as input and a mixed complex, i.e. a chain complex equipped with an additional degree $+1$ operator squaring to zero, as output. The coHochschild complex is invariant with respect to a suitable notion of weak equivalence drawn from Koszul duality theory. 
\\

\item The fourth observation is that in order to recover a homological model for the free loop space $L|X|$ for an arbitrary simplicial set $X$, certain formal localization must be performed at the $1$-simplices of $X$. This step is not necessary if all the $1$-simplices in $X$ already have inverses up to homotopy (e.g. $X$ is a Kan complex). This localization may be described in purely algebraic terms. In order to extract the set of elements to be localized, we consider categorical coalgebras equipped with an additional dg coalgebra enrichment on their associated dg category. We call these $B_{\infty}$-\textit{categorical coalgebras}. We define an extended version of the coHochschild complex that takes a $B_{\infty}$-categorical coalgebra and produces a chain complex by formally inverting a particular set (extracted by applying the \textit{set-like elements functor} to the dg coalgebra enrichment of the associated cobar dg category) in the coHochschild complex of the underlying categorical coalgebra.
\\

\end{enumerate}
Our main result, informally stated, is the following.
\begin{theorem*}
For an arbitrary simplicial set $X$, the extended coHochschild complex of $C_*(X)$, a $B_{\infty}$-categorical coalgebra model for the normalized chains on $X$, is naturally quasi-isomorphic to $C^{\text{sing}}_*(L|X|)$, the singular chains on the free loop space.
\end{theorem*}
As our proof will reveal, this result may be understood as a simplification of a theorem proved by Goodwillie in \cite{Goodwillie} and independently by Burghelea and Fedorowicz in \cite{BF} saying the following. For any path-connected pointed topological space $(Y,b)$, the Hochschild chain complex of the Pontryagin dg algebra $C_*^{\text{sing}}(\Omega_bY)$ of singular chains on the based (Moore) loop space of $Y$ at $b$ is naturally quasi-isomorphic to $C^{\text{sing}}_*(LY)$. Our streamlined model is essentially deduced from this result using (a generalization of) the fact that for any conilpotent dg coalgebra $C$ we have two resolutions for the dg algebra $A= \text{Cobar}(C)$ as an $A$-bimodule: 1) the classical two-sided bar resolution $\text{Bar}(A,A,A)$, and 2) a smaller resolution $\mathcal{Q}(A,C,A)$ with underlying module $A \otimes C \otimes A$. The first one is used when defining the Hochschild complex of $A$, while the second one is used when defining the coHochschild complex of $C$. We describe an explicit natural quasi-isomorphism of $A$-bimodules 
\[\mathcal{Q}(A,C,A) \xrightarrow{\simeq} \text{Bar}(A,A,A), \]
see Proposition \ref{contraction}. 

Finally, we establish a relationship between the extended coHochschild complex model for the free loop space and Brown's twisted tensor product model for a fibration. This involves proving that a natural dg bialgebra structure constructed on the extended cobar construction of a reduced simplicial set is in fact a dg Hopf algebra. We then model the holonomy of the free loop space fibration in terms of the adjoint action of such dg Hopf algebra. This relationship with Brown's twisted tensor product may be used to give an algebraic model of the inclusion $Y \to LY$ of points as constant loops in terms of the coHochschild complex. We expect this to be useful in studying and computing the string topology of non-simply connected manifolds. 
\begin{ack} This research was supported by NSF Grant 210554 and the Karen EDGE Fellowship. The author would like to thank Julian Holstein, Ralph Kaufmann, Andrey Lazarev, Anibal Medina, Yang Mo, Maximilien Peroux, Dennis Sullivan, Daniel Tolosa, and Mahmoud Zeinalian, for fruitful exchanges. 
\end{ack}

\section{Preliminaries}
Fix a commutative ring with unit $\mathbf{k}$. We assume familiarity with the notions of differential graded (dg) $\mathbf{k}$-modules, $\mathbf{k}$-algebras, $\mathbf{k}$-coalgebras and $\mathbf{k}$-categories. For generalities about dg categories and their homotopy theory we refer to \cite{tabuada1}, \cite{tabuada2}, and \cite{toen}. All (co)algebras in this article will be (co)associative and (co)unital. All differentials will have degree $-1$. We denote by $\mathsf{Ch}_{\mathbf{k}}$ the category of dg $\mathbf{k}$-modules (i.e. $\mathbf{k}$-chain complexes) and $\mathsf{Ch}_{\mathbf{k}}^{\geq 0}$ its full sub-category of non-negatively graded objects. Denote by $\mathsf{dgAlg}_{\mathbf{k}}$, $\mathsf{dgCoalg}_{\mathbf{k}}$, and $\mathsf{dgCat}_{\mathbf{k}}$ the categories of dg algebras, dg coalgebras, and small dg categories, respectively. \textit{In this article, we will furthermore assume that all dg algebras and coalgebras are flat as $\mathbf{k}$-modules and all dg categories are locally $\mathbf{k}$-flat.} An additional subscript of ``$\geq 0$" in the notation for these categories will also mean the full sub-category of non-negatively graded objects. Whenever we write $\otimes$ we mean $\otimes_{\mathbf{k}}$, unless noted otherwise. All signs in this article are determined by the Koszul sign convention.

For any graded algebra $A$, we denote by $A^{op}$ the graded algebra with $A$ as underlying  $\mathbf{k}$-module and multiplication defined by $\mu_A \circ t$, where
$t: A \otimes A \to A \otimes A$ is given by $t(a \otimes b) = (-1)^{|a||b|} b \otimes a$ and $\mu \colon A \otimes A \to A$ is the multiplication of $A$. Similarly, for any graded coalgebra $C$, we denote by $C^{op}$ the graded coalgebra with $X$ as underlying $\mathbf{k}$-module and coproduct defined by $t \circ \Delta$ where $\Delta  \colon C \to C \otimes C$ is the coproduct of $C$.

For any set $S$ we denote by $\mathbf{k}[S]$ the $\mathbf{k}$-coalgebra whose underlying $\mathbf{k}$-module is freely generated by $S$ and whose coproduct $\Delta \colon \mathbf{k}[S] \to \mathbf{k}[S]\otimes \mathbf{k}[S]$ is determined by $\Delta(s)=s \otimes s$ for any $s \in S$. The counit $\varepsilon \colon \mathbf{k}[S]\to \mathbf{k}$ is detrmined by $\varepsilon(s)=1_{\mathbf{k}}$ for any $s \in S.$

\subsection{Cotensor product}

Let $C$ be a dg $\mathbf{k}$-coalgebra. Let $M$ and $N$ be dg right and left $C$-comodules, respectively, with coaction maps
$\rho_M\colon M \to M \otimes C$ and $\rho_N\colon N \to C \otimes N$. The \textit{cotensor product of $M$ and $N$ over $C$} is defined as
\begin{eqnarray}
M \underset{C}{\msquare} N = \text{ker}( \rho_M \otimes \text{id}_N - \text{id}_M \otimes \rho_N \colon M \otimes N \to M \otimes C \otimes N). 
\end{eqnarray}

The category of $\mathbf{k}$-flat dg $C$-bicomodules, denoted by $C$-$\mathsf{biComod}$, becomes a monoidal category when equipped with the cotensor product $\square_C$ and unit object $C$. Suppose $A$ is a monoid in this category, namely, a dg $C$-bicomodule equipped with an associative product $A \square_C A \to A$ and unit $u: C \to A$. Let $E$ and $F$ be dg right and left $C$-comodules, respectively. Suppose that $E$ and $F$ are further equipped with right and left dg $A$-module structures respectively, namely, we have action maps
$\rho_E: E \square_C A \to E$ and $\rho_F: A \square_C F \to F$ satisfying the usual compatibilities. Define the \textit{tensor product of $E$ and $F$ over $A$} to be the dg $\mathbf{k}$-module 
\begin{eqnarray}\label{tensorproduct} E \underset{A}{\bigotimes} F = \text{coker}(\rho_E \underset{C}{\square} \text{id}_F -\text{id}_E \underset{C}{\square} \rho_F \colon  E \underset{C}{\square}  A \underset{C}{\square}  F \to E \underset{C}{\square} F).
\end{eqnarray}

\subsection{DG categories as monoids} Given any dg category $\mathsf{A} \in \mathsf{dgCat}_{\mathbf{k}}$ with object set $\mathcal{O}_{\mathsf{A}}$ define a monoid $\mathcal{M}(\mathsf{A})$ in the monoidal category of dg $\mathbf{k}[\mathcal{O}_{\mathsf{A}}]$-bicomodules equipped with the cotensor product $\square_{\mathbf{k}[\mathcal{O}_{\mathsf{A}}]}$ as follows. The underlying dg $\mathbf{k}$-module of $\mathcal{M}(\mathsf{A})$ is given by the direct sum 
\[ \bigoplus_{x,y \in \mathcal{O}_{\mathsf{A}}}\mathsf{A}(x,y).\]
The $\mathbf{k}[\mathcal{O}_{\mathsf{A}}]$-bicomodule structure maps
\[ \mathcal{M}(\mathsf{A}) \to \mathbf{k}[\mathcal{O}_{\mathsf{A}}] \otimes \mathcal{M}(\mathsf{A}) \] and
\[\mathcal{M}(\mathsf{A}) \to \mathcal{M}(\mathsf{A}) \otimes \mathbf{k}[\mathcal{O}_{\mathsf{A}}] 
\]
are induced by the source and target maps in $\mathsf{A}$, respectively.
The monoid structure 
\[\mathcal{M}(\mathsf{A}) \underset{\mathbf{k}[\mathcal{O}_{\mathsf{A}}]}{\msquare} \mathcal{M}(\mathsf{A}) \to \mathcal{M}(\mathsf{A}) \]
is induced by the composition of morphisms in $\mathsf{A}$ and the unit map
$\mathbf{k}[\mathcal{O}_{\mathsf{A}}]\to \mathcal{M}(\mathsf{A})$ is determined by $x \mapsto \text{id}_x \in \mathsf{A}(x,x)_0$ for all $x \in \mathcal{O}_{\mathsf{A}}.$

\subsection{Two sided bar construction} \label{barconstruction}
Let $\mathsf{A}$ be a dg category and, for simplicity, denote by $\mathsf{A}_0$ the coalgebra $\mathbf{k}[\mathcal{O}_{\mathsf{A}}].$ Let $\mathsf{M}$ and $\mathsf{N}$ be right and left dg modules over $\mathcal{M}(\mathsf{A})$, respectively, in the monoidal category $(\mathsf{A}_0\text{-}\mathsf{biComod}, \square_{\mathsf{A}_0})$. This means that $\mathsf{M}$ and $\mathsf{N}$ are $\mathbf{k}$-flat dg $\mathsf{A}_0$-bicomodules equipped dg maps
$$ \mathsf{M} \square_{\mathsf{A}_0} \mathcal{M}(\mathsf{A}) \to \mathsf{M} \text{ and } \mathcal{M}(\mathsf{A}) \square_{\mathsf{A}_0} \mathsf{N}\to \mathsf{N}$$
defining right and left dg $\mathcal{M}(\mathsf{A})$-actions, respectively. The \textit{two-sided bar construction of $\mathsf{M}$ and $\mathsf{N}$ over $\mathsf{A}$} is the dg $\mathsf{A}_0$-bicomodule
\[ \text{Bar}_{\mathsf{A}_0}(\mathsf{M}, \mathsf{A}, \mathsf{N}) \]
defined as follows. The underlying graded $\mathbf{k}$-module is defined to be
\[\bigoplus_{i=0}^{\infty} \Big(\mathsf{M} \underset{\mathsf{A}_0}{\msquare} \big(s^{+1}\overline{\mathcal{M}}(\mathsf{A})\big)^{\square i}  \underset{\mathsf{A}_0}{\msquare}\mathsf{N}\Big), \]
where $\overline{\mathcal{M}}(\mathsf{A})= \mathcal{M}(\mathsf{A}) / u(\mathsf{A_0})$, where $u: \mathsf{A}_0 \to \overline{\mathcal{M}}(\mathsf{A})$ is the unit map. 

We will use the classical ``bar" notation 
\[ m  [a_1 \big| \cdots \big|a_p]  n\]
to denote a generator 
\[m \square s^{+1}a_1 \square \cdots \square s^{+1}a_p \square n, \] 
where $m \in \mathsf{M}$, $n \in \mathsf{N}$, and $a_i \in \mathcal{M}(\mathsf{A})$ for $i=1,\cdots, p.$ 
The differential
\[\partial_{\mathsf{M},\mathsf{A},\mathsf{N}} \colon \bigoplus_{i=0}^{\infty} \Big(M \underset{\mathsf{A}_0}{\msquare} \big(s^{+1} \overline{\mathcal{M}}(\mathsf{A})\big)^{\square i}  \underset{\mathsf{A}_0}{\msquare} N\Big) \to \bigoplus_{i=0}^{\infty} \Big(M \underset{\mathsf{A}_0}{\msquare} \big(s^{+1}\overline{\mathcal{M}}(\mathsf{A})\big)^{\square i}  \underset{\mathsf{A}_0}{\msquare} N\Big) \]
is defined as the sum of linear maps
\[\partial_{\mathsf{M},\mathsf{A},\mathsf{N}}= d_{\mathsf{M}} \square \text{id}_{\mathcal{M}(\mathsf{A})} \square \text{id}_{\mathsf{N}} + \text{id}_{\mathsf{M}} \square D_{\mathcal{M}(\mathsf{A})}\square \text{id}_{\mathsf{N}} + \text{id}_{\mathsf{M}}  \square \text{id}_{\mathcal{M}(\mathsf{A})} \square d_{\mathsf{N}} + \theta, \]
where $d_{\mathsf{M}}, d_{\mathsf{N}}$, and $D_{\mathcal{M}(\mathsf{A})}$ are the differentials of $\mathsf{M}, \mathsf{N}$, and $\mathcal{M}(\mathsf{A})$, respectively, and $\theta$ is given by the following formula
\[\theta( m [a_1 \big| \cdots \big| a_p]n)= m\cdot a_1 [a_2\big| \cdots \big| a_p] n + \sum_{i=1}^{p-1} \pm m [a_1\big| \cdots \big| a_i \cdot a_{i+1} \big| \cdots \big| a_p]n \pm m [a_1\big| \cdots \big| a_{p-1}]a_p\cdot m. \]
The associativity of the monoid structure of $\mathcal{M}(\mathsf{A})$, the compatibilities of the differentials with the products and actions, and $d_M^2=d_N^2= D_{\mathcal{M}(\mathsf{A})}^2=0$ all together imply that $\partial_{\mathsf{M},\mathsf{A},\mathsf{N}}^2=0$. The $\mathsf{A}_0$-bicomodule structure on $\bigoplus_{i=0}^{\infty} \Big(M \underset{\mathsf{A}_0}{\msquare} \big(s^{+1}\overline{\mathcal{M}}(\mathsf{A})\big)^{\square i}  \underset{\mathsf{A}_0}{\msquare} N\Big)$ is given by the left and right $\mathsf{A}_0$-comodule structures of $\mathsf{M}$ and $\mathsf{N}$, respectively.
\subsection{The Hochschild complex}
For any dg category $\mathsf{A}$, the chain complex 
\[\text{Bar}_{\mathsf{A}_0}(\mathcal{M}(\mathsf{A}), \mathsf{A}, \mathcal{M}(\mathsf{A}) )\] has a natural dg $\mathcal{M}(\mathsf{A})$-bimodule structure in the category of $\mathsf{A}_0$-bicomodules with cotensor product $\square_{\mathsf{A}_0}.$ This construction is clearly functorial with respect to morphisms of dg categories. We recall the definition of the Hochschild chain complex. 

\begin{definition} Define a functor
\[\mathcal{CH}_*\colon \mathsf{dgCat}^{\geq 0}_{\mathbf{k}} \to \mathsf{Ch}_{\mathbf{k}}^{\geq 0}, \]
called the \textit{Hochschild complex}, as follows. For any $\mathsf{A} \in \mathsf{dgCat}^{\geq 0}_{\mathbf{k}}$, the underlying dg $\mathbf{k}$-module of $\mathcal{CH}_*(\mathsf{A})$ is defined by 
\[ 
\text{Bar}_{\mathsf{A}_0}(\mathcal{M}(\mathsf{A}), \mathsf{A}, \mathcal{M}(\mathsf{A}) ) \underset{\mathcal{M}(\mathsf{A}) \otimes \mathcal{M}(\mathsf{A})^{op}}{\bigotimes} \mathcal{M}(\mathsf{A}) ,\]
see \ref{tensorproduct} for notation. 

The generators of $\mathcal{CH}_*(\mathsf{A})$ may be written as $[a_1| \cdots |a_p]a_{p+1},$ where $a_{p+1} \in \mathcal{M}(\mathsf{A})$, $a_i \in \overline{\mathcal{M}}(\mathsf{A})$ with $\mathsf{s}(a_i)=\mathsf{t}(a_{i+1})$ for $i=1\cdots p$, and $\mathsf{t}(a_{p+1})=\mathsf{s}(a_p)$. Using this notation, the differential 
\[ \partial_{\mathsf{A}} \colon \mathcal{CH}_*(\mathsf{A}) \to \mathcal{CH}_{*-1}(\mathsf{A})\] is given by the same formula as the differential for the Hochschild complex of a dg algebra. One may equip this construction with a mixed complex structure via Connes' operator
\[B: \mathcal{CH}_*(\mathsf{A}) \to \mathcal{CH}_{*+1} (\mathsf{A}).\] In this setting, $B$ is given by
\[B([a_1| \cdots |a_p]a_{p+1}) =  \sum_{i=1}^{p+1} \pm [a_{i}| \cdots | a_{p+1}| a_1 | \cdots |a_{i-1}]\text{id}_{\mathsf{s}(a_{i})}s \]
Just as in the classical case of the Hochschild complex of a dg algebra, one may check that $(\bigoplus_{n=0}^{\infty} \mathcal{CH}_n(\mathsf{A}), \partial_{\mathsf{A}}, B)$ is a non-negatively graded mixed complex functorially associated to any $\mathsf{A} \in \mathsf{dgCat}^{\geq 0}_{\mathbf{k}}.$ 
\end{definition}

\section{Categorical coalgebras and the cobar construction}
In this section we define the notion of categorical coalgebras. This is a version of a curved coalgebra over a set of ``objects" or ``points" satisfying certain properties. Any categorical coalgebra gives rise to a dg category through a many object version of the cobar construction. These notions have been adapted from \cite{holstein-lazarev} in order to be applied to the algebraic topology setting and to work over an arbitrary commutative ring $\mathbf{k}$. The corresponding notion in \cite{holstein-lazarev} is that of a ``pointed curved coalgebra". 

\subsection{Categorical coalgebras} 

\begin{definition} \label{categoricalcoalgebra}
A \textit{categorical $\mathbf{k}$-coalgebra} consists of the data $C=(C, \Delta, d, h)$
such that
\begin{enumerate}
    \item $C= \bigoplus_{i=0}^{\infty} C_i$  is a non-negatively graded flat $\mathbf{k}$-module.
    \item \label{def2} $\Delta: C \to C \otimes C$ is a degree $0$ coassociative counital coproduct with counit $\varepsilon: C \to \mathbf{k}$
    \item The set \[\mathcal{S}(C):= \{ x \in C : \Delta(x)= x\otimes x, \varepsilon(x)=1_{\mathbf{k}} \}\] of
``set-like" elements in $C$ is non-empty and \[C_0 \cong \mathbf{k}[\mathcal{S}(C)].\]
    
    \item \label{def4} $d: C \to C$ is a linear map of degree $-1$ which is a graded coderivation of $\Delta.$
    \item \label{def5} The projection map $\epsilon: C \to C_0$ satisfies $\epsilon \circ d=0.$ In other words, $d: C_1 \to C_0$ is the zero map. 
    \item \label{def6} $h: C \to \mathbf{k}$ is a linear map of degree $-2$ satisfying $h \circ d =0$ and 
    \begin{eqnarray}\label{curvature}
    d\circ d= (h \otimes \text{id}) \circ (\Delta - \Delta^{op})
    \end{eqnarray}
    where $\Delta^{op}= t \circ \Delta$ for $t(x \otimes y)= (-1)^{|x||y|} y \otimes x$. The right hand side of the above equation is being considered as a map $C \to \mathbf{k} \otimes C \cong C.$ The map $h$ is called the \textit{curvature} of $C$. Equation \ref{curvature} may be rewritten as
    \[d^2(x) = \sum_{(x)} h(x')x'' + x'h(x'').\]

\end{enumerate}
\end{definition}

Any categorical coalgebra $C$ has a natural $C_0$-bicomodule structure with coaction maps
$$\rho_l: C \xrightarrow{\Delta} C \otimes C \xrightarrow{\epsilon \otimes \text{id}_C} C_0 \otimes C$$
and
$$\rho_r: C \xrightarrow{\Delta} C \otimes C \xrightarrow{\text{id}_C \otimes \epsilon} C \otimes C_0.$$
Furthermore, the coassociativity of $\Delta: C \to C\otimes C$ implies that $\Delta: C \to C \otimes C$ factors as $C \to C \square_{C_0} C \to C \otimes C$; so $C$ may be regarded as a comonoid in the category of graded $C_0$-bicomodules with monoidal structure given by the cotensor product $\square_{C_0}.$

\begin{remark} Note that it is possible for a categorical coalgebra to have non-zero cuvature and for $d: C \to C$ to square zero.
\end{remark}

\begin{definition}\label{morphism} A \textit{morphism of categorical coalgebras} $C=(C, \Delta, d, h)$ and  $C'=(C', \Delta', d', h')$  consists of a pair $(f_0, f_1)$ where
\begin{enumerate}
\item $f_0: (C, \Delta) \to (C', \Delta')$ is a morphism of graded $\mathbf{k}$-coalgebras, 
\item $f_1: C \to C_0'$ is a $C_0'$-bicomodule map of degree $-1$ such that the composition $\bar{f_1}= \varepsilon \circ f_1$, where $\varepsilon'$ is the counit of $C'$, satisfies
\begin{eqnarray}\label{morphism1}
f_0 \circ d = d' \circ f_0 + (\bar{f_1} \otimes f_0) \circ (\Delta - \Delta^{op})
\end{eqnarray}
and
\begin{eqnarray}\label{morphism2}
h' \circ f_0= h + \bar{f_1} \circ d + (\bar{f_1} \otimes \bar{f_1}) \circ \Delta.
\end{eqnarray}

\end{enumerate}
The composition of two morphisms of categorical coalgebras is defined by  
\begin{eqnarray}\label{composition}
(g_0, g_1) \circ (f_0, f_1)= (g_0 \circ f_0 , g_1 \circ f_0 + g_0 \circ f_1).
\end{eqnarray}
Denote by $\mathsf{cCoalg}_{\mathbf{k}}$ the category of categorical coalgebras. 
\end{definition}

\subsection{The cobar functor} \label{cobarconstruction}
Working over a field and with unbounded complexes, Holstein and Lazarev define in \cite{holstein-lazarev} a functor from pointed curved coalgebras to dg categories extending the classical cobar functor from conilpotent dg coalgebras to augmented dg algebras. The same construction can be defined for categorical coalgebras over an arbitrary ring. We now describe this construction in our setting. 
\begin{definition}
Define a functor
\[\mathbf{\Omega} \colon \mathsf{cCoalg}_{\mathbf{k}} \to \mathsf{dgCat}_{\mathbf{k}}^{\geq 0},\]
called the \textit{cobar functor}, as follows. Given any $C= (C, d, \Delta, h) \in \mathsf{cCoalg}_\mathbf{k}$, the objects of $\mathbf{\Omega}(C)$ are the elements of the set $\mathcal{S}(C)$ of set-like elements in $C$.

For any $x \in \mathcal{S}(C)$ denote by $i_x: \mathbf{k} \to C_0=\mathbf{k}[\mathcal{S}(C)]$ the map determined by $i_x(1_\mathbf{k})= x$. The map $i_x$ gives rise to a $C_0$-bicomodule structure on $\mathbf{k}$ through the  maps $$\mathbf{k} \cong \mathbf{k} \otimes \mathbf{k} \xrightarrow{i_x \otimes \text{id}_{\mathbf{k}}} C_0 \otimes \mathbf{k}$$ and $$\mathbf{k} \cong \mathbf{k} \otimes \mathbf{k} \xrightarrow{\text{id}_\mathbf{k} \otimes i_x} \mathbf{k} \otimes C_0.$$ We denote this $C_0$-bicomodule by $\mathbf{k}_x$ and its generator by $\text{id}_x.$

Write $C= \overline{C} \oplus C_0$. Denote by $s^{-1}\overline{C}$ the graded $\mathbf{k}$-module obtained by applying the shifting $\overline{C}$ by $-1$. We have the following three degree $-1$ maps: 
\begin{enumerate}
\item $\overline{d}\colon s^{-1}\overline{C} \to s^{-1}\overline{C}$
\item $\overline{\Delta}\colon s^{-1}\overline{C} \to s^{-1}\overline{C} \otimes s^{-1}\overline{C}, \text{ and }$
\item $\overline{h}\colon s^{-1}\overline{C} \xrightarrow{s^{+1}} \overline{C} \xrightarrow{\rho_r} C \otimes C_0 \xrightarrow{h \otimes \textit{id}} \mathbf{k} \otimes C_0 \cong C_0.$
\end{enumerate}

For any two $x,y \in \mathcal{S}(C)$ define a non-negatively graded $\mathbf{k}$-module by
$$\mathbf{\Omega}(C)(x,y)= \bigoplus_{i=0}^{\infty} \mathbf{k}_x \underset{C_0}{\msquare} (s^{-1}\overline{C})^{\square i} \underset{C_0}{\msquare} \mathbf{k}_y,$$
where $(s^{-1}\overline{C})^{\square i} $ denotes the $i$-fold cotensor product of $C_0$-bicomodules and $(s^{-1}\overline{C})^{\square 0}=C_0$.
We will use the notation
\[ \{c_1| \cdots |c_p\}\]
to denote a generator
\[ \text{id}_x \square s^{-1}c_1 \square \cdots \square s^{-1}c_p \square \text{id}_y \in \mathbf{\Omega}(C)(x,y). \] We say the monomial $\{c_1| \cdots |c_p\}$ has \textit{length} $p$. In particular, note 
\[ \mathbf{\Omega}(C)(x,x)_0 = \mathbf{k}_x \oplus \mathbf{k}_x \square s^{-1}C_1 \square \mathbf{k}_x .\]

The differential $$D_{x,y}: \mathbf{\Omega}(C)(x,y)_k \to \mathbf{\Omega}(C)(x,y)_{k-1}$$ is defined by extending
\[\overline{h} + \overline{d} + \overline{\Delta} \colon \mathbf{k}_x \square s^{-1}\overline{C} \square \mathbf{k}_y \to
\mathbf{k}_x \square C_0 \square \mathbf{k}_y \oplus
\mathbf{k}_x \square s^{-1}\overline{C} \square \mathbf{k}_y \oplus \mathbf{k}_x \square (s^{-1}\overline{C})^{\square 2} \square \mathbf{k}_y
\]
as a ``derivation" to monomials of arbitrary length. It follows directly from (\ref{def2}), (\ref{def4}), and (\ref{def6}) in Definition \ref{categoricalcoalgebra} that $D_{x,y} \circ D_{x,y}=0$. %see Lemma 3.18 in \cite{holstein-lazarev} (even if they work over a field, this hypothesis is not used in the argument).%
The composition in $\mathbf{\Omega}(C)$ is given by concatenation of monomials. For every $x \in \mathcal{S}(C)$, $1_{\mathbf{k}} \in  \mathbf{k}_x \cong \mathbf{k}_x \square C_0 \square \mathbf{k}_x \subset \mathbf{\Omega}(C)(x,x)_0$ is the identity morphism.

Given a morphism $(f_0,f_1): C \to C'$ between categorical coalgebras, define a morphism
\[\mathbf{\Omega}(f_0,f_1)\colon \mathbf{\Omega}(C) \to \mathbf{\Omega}(C') \]
of dg categories as follows. Since $f_0: C\to C'$ is a map of coalgebras, $f_0$ restricts to a map of sets $\mathcal{S}(C) \to \mathcal{S}(C')$, which defines the functor $\mathbf{\Omega}(f_0,f_1)$ on objects. For any two $x,y \in \mathcal{S}(C) $
define 
\[ \mathbf{\Omega}(f_0,f_1)_{x,y} \colon \mathbf{\Omega}(C)(x,y) \to \mathbf{\Omega}(C')(f_0(x), f_0(y)) \]
by extending the map 
\begin{eqnarray} \mathbf{k}_x  \square  s^{-1} \overline{C} \square \mathbf{k}_y \longrightarrow \mathbf{k}_{f_0(x)} \square s^{-1} \overline{C} \square \mathbf{k}_{f_0(y)} \oplus \mathbf{k}_{f_0(x)} \square C_0' \square  \mathbf{k}_{f_0(y)}
\\
\{c\} \longmapsto \{f_0(c)\} + \text{id}_{f_0(x)} \square f_1(c) \square \text{id}_{f_0(y)}
\end{eqnarray}
``multiplicatively" to monomials $\{c_1 | \cdots | c_p \}$ of arbitrary length. 
Note that \[ \mathbf{k}_{f_0(x)} \square C' _0 \square  \mathbf{k}_{f_0(y)}\] is a non-trivial $\mathbf{k}$-module if and only if $f_0(x)=f_0(y)$, in which case it is isomorphic to $\mathbf{k}$. Hence, $\text{id}_{f_0(x)} \square f_1(c) \square \text{id}_{f_0(y)}$ may be identified with a scalar. It follows directly from \ref{morphism1} and \ref{morphism2} that $\mathbf{\Omega}(f_0,f_1)_{x,y}$ is a chain map for each $x,y \in \mathcal{O}_{C}$ and from \ref{composition} that compositions are compatible.

\end{definition}
\begin{remark} When $\mathbf{k}$ is a field, a categorical $\mathbf{k}$-coalgebra is a pointed curved $\mathbf{k}$-coalgebra $C$ (as defined in \cite{holstein-lazarev}) that is non-negatively graded and whose coradical is exactly the degree zero summand $C_0 \subseteq C$ (which is assumed to be non-trivial). In this case, the splitting map (which is part of the structure in the definition of a pointed curved coalgebra) is precisely the projection map $C \to C_0$. In particular, by the definition of a pointed curved coalgebra, $C_0$ is generated by the set-like (sometimes called ``group-like") elements of $C$. 
\end{remark}
\begin{definition}
A $B_{\infty}$-\textit{categorical coalgebra} is a categorical coalgebra $C$ equipped with degree $0$ coassociative coproducts
\[\nabla_{x,y}: \mathbf{\Omega}(C)(x,y) \to \mathbf{\Omega}(C)(x,y) \otimes \mathbf{\Omega}(C)(x,y)\]
for all $x,y \in \mathcal{S}(C)$ making $\mathbf{\Omega}(C)$ into a category enriched over $(\mathsf{dgCoalg}_{\mathbf{k}}^{\geq 0}, \otimes)$, the monoidal category of differential non-negatively graded coassociative counital $\mathbf{k}$-coalgebras. $B_{\infty}$-categorical coalgebras form a category when equipped with maps of categorical coalgebras that preserve the additional structure. We denote this category by $\mathsf{B}_{\infty}\text{-}\mathsf{cCoalg}_{\mathbf{k}}$. This notion has also been considered in \cite{MRZ}.
\end{definition}
\subsection{The extended cobar functor}
We define a new version of the cobar construction by formally inverting set-like elements in the dg coalgebra of morphisms of the cobar construction of a  $B_{\infty}$-categorical coalgebra, generalizing a construction of \cite{hess2010loop}. This will give rise to a functorial construction that recovers the dg category of paths of the geometric realization of a simplicial set $X$ when applied to a $B_{\infty}$-categorical coalgebra of chains on $X$, as it will be discussed in Section \ref{chainsandpaths}. 

Let \[\mathcal{Z} \colon \mathsf{B_{\infty}\text{-}cCoalg}_{\mathbf{k}} \to \mathsf{Cat}\] be the functor defined as follows. For any $C\in  \mathsf{B_{\infty}\text{-}cCoalg}_{\mathbf{k}}$, the set of objects of $\mathcal{Z}(C)$ is $\mathcal{S}(C)$. For any two objects $x$ and $y$
\[\mathcal{Z}(C)(x,y)= \{ f \in 
\mathbf{\Omega}(C)(x,y) | \nabla_{x,y}(f) = f \otimes f \text{ and } \varepsilon_{x,y}(f)=1 \} .\] In other words, $\mathcal{Z}(C)(x,y)=\mathcal{S}(\mathbf{\Omega}(C)(x,y), \nabla_{x,y}).$  Since the composition in $\mathbf{\Omega}(C)$ is compatible with the dg coalgebra structures on the morphisms and identity morphisms are set-like, it follows that $\mathcal{Z}(C)$ becomes a category with composition induced by that of $\mathbf{\Omega}(C)$. The functoriality of the construction follows since $\mathbf{\Omega}$ is a functor and taking set-like elements in a coalgebra is functorial. 

\begin{definition}
Define a functor
\[\widehat{\mathbf{\Omega}} \colon \mathsf{B_{\infty}\text{-}cCoalg}_{\mathbf{k}} \to \mathsf{dgCat}_{\mathbf{k}}^{\geq 0},\]
called the \textit{extended cobar functor}, by letting
\[ \widehat{\mathbf{\Omega}}(C)=  \mathbf{\Omega}(C)[ \mathcal{Z}(C)^{-1}],\]
namely, by formally (strictly) inverting the set of $0$-cycles determined by the morphisms of $\mathcal{Z}(C)$ inside $\mathbf{\Omega}(C)$.
\begin{remark} \label{derivedloc} In practice, we will consider the above construction when the natural map \[\mathbf{k}[\mathcal{Z}(C)] \to \mathbf{\Omega}(C)\]
is a cofibration of cofibrant and locally $\mathbf{k}$-flat dg categories, where $\mathbf{k}[\mathcal{Z}(C)]$ denotes the dg category obtained by linearizing the morphisms of $\mathcal{Z}(C)$ and defining each differential to be trivial.

In this case, the strict localization \[ \widehat{\mathbf{\Omega}}(C)=  \mathbf{\Omega}(C)[ \mathcal{Z}(C)^{-1}]\] is a homotopical localization. In other words, under these hypotheses, an $\mathbf{\Omega}$-quasi-equivalence $f: C \to C'$ between categorical coalgebras induces a quasi-equivalence 
$\widehat{\mathbf{\Omega}}(f) \colon \widehat{\mathbf{\Omega}}(C) \to \widehat{\mathbf{\Omega}}(C')$ of dg categories. This follows since we may interpret the extended cobar functor as a pushout of dg categories
\[\widehat{\mathbf{\Omega}}(C) = \mathbf{\Omega}(C)\underset{\mathbf{k}[\mathcal{S}(C)]}{\bigsqcup} \mathbf{k}[\mathcal{Z}(C)][\mathcal{Z}(C)^{-1}]. \] 
This pushout is a homotopy pushout in Tabuada's model structure on $\mathsf{dgCat}_{\mathbf{k}}$ when $\mathbf{k}[\mathcal{Z}(C)] \to \mathbf{\Omega}(C)$ is a cofibration of dg categories, since both $\mathbf{k}[\mathcal{Z}(C)]$ and \newline $\mathbf{k}[\mathcal{Z}(C)][\mathcal{Z}(C)^{-1}]$ are locally $\mathbf{k}$-flat dg categories and consequently left proper objects, see \cite{holstein-proper}.

\end{remark}

\end{definition}

\section{Chains on a simplicial set and the dg category of paths} \label{chainsandpaths}

The first goal of this section is to describe a version of the normalized simplicial chains as a functor
\[C_* \colon \mathsf{sSet} \to \mathsf{B_{\infty}\text{-}cCoalg}_{\mathbf{k}}.\]
Then we show that, for any simplicial set $X$, the extended cobar construction applied to the $\mathsf{B}_{\infty}$-categorical coalgebra of chains $C_*(X)$ yields a model for the dg category of paths on the topological space $|X|$.

\subsection{Chains as a categorical coalgebra}
For any simplicial set $X$, denote by  $(\overline{C}_*(X), \partial)$ the dg $\mathbf{k}$-module of normalized simplicial chains. Recall the Alexander-Whitney coproduct, given on any simplex $\sigma \in X_n$ by
\[\Delta(\sigma) = \sum_{i=0}^n \sigma(0,...,i) \otimes \sigma(i,...,n),\]
induces a coassociative coproduct \[\Delta: \overline{C}_*(X) \to \overline{C}_*(X) \otimes \overline{C}_*(X)\] of degree $0$. In the above formula, $\sigma(0,...,i)$ and $\sigma(i,...,n)$ denote the first $i$-th and last $(n-i)$-th faces of $\sigma,$ respectively. This construction gives rise to a functor 
\[C^{\Delta}_*: \mathsf{sSet} \to \mathsf{dgCoalg}^{\geq 0}_{\mathbf{k}} \] given by 
\[C^{\Delta}_*(X)=(\overline{C}_*(X), \partial, \Delta)\]

For any two simplicial sets $X$ and $Y$, the natural Eilenberg-Zilber shuffle map
$$\text{EZ}_{X,Y}: \overline{C}_*(X) \otimes \overline{C}_*(Y) \to \overline{C}_*(X \times Y)$$
is a map of dg coalgebras and consequently makes $C^{\Delta}_*$ into a lax monoidal functor, as explained in 17.6 of \cite{EilMoo66}.

The projection map $\epsilon: \overline{C}_*(X) \to \overline{C}_0(X)$ does not satisfy $\epsilon \circ \partial=0.$ However, as suggested in \cite{holstein-lazarev}, the differential $\partial$ may be modified to obtain a categorical coalgebra as follows.

\begin{definition} \label{chains} For any $X \in \mathsf{sSet}$ define a categorical coalgebra $\widetilde{C}^{\Delta}_*(X) \in \mathsf{cCoalg}_{\mathbf{k}}$ as follows.
The underlying graded $\mathbf{k}$-module of  $\widetilde{C}^{\Delta}_*(X)$ is exactly $\overline{C}_*(X)$, which is given by $\overline{C}_n(X)= \mathbf{k}[X_n]/D(X_n)$, where $D(X_n)\subseteq \mathbf{k}[X_n]$ is the sub $\mathbf{k}$-module  generated by degenerate $n$-simplices. 

Let $e: \mathbf{k}[X_1] \to \mathbf{k}$ be the linear map sending degenerate $1$-simplices to $0 \in \mathbf{k}$ and non-degenerate $1$-simplices to $1 \in \mathbf{k}$. The map $e$ induces a linear map $\widetilde{e}: \overline{C}_1(X) \to \mathbf{k}.$ Define a new differential
\[\widetilde{\partial}: \overline{C}_*(X) \to \overline{C}_{*-1}(X)\] 
by \[\widetilde{\partial}= \partial - (\text{id}\otimes \widetilde{e} - \widetilde{e}\otimes\text{id}) \circ \Delta.\]

The map $\widetilde{\partial}$ is a coderivation of $\Delta$ and the projection map $\epsilon: \overline{C}_*(X) \to \overline{C}_0(X)$ now satisfies $\epsilon \circ \widetilde{\partial}=0$. Finally, define $h: \overline{C}_2(X) \to \mathbf{k}$ by 
\[h= (\widetilde{e} \otimes \widetilde{e}) \circ \Delta + \widetilde{e} \circ \partial.
\]

A routine check yields that
\[\widetilde{C}^{\Delta}_*(X)=(\overline{C}_*(X), \widetilde{\partial}, \Delta, h)\]
defines an object in $\mathsf{cCoalg}_\mathbf{k}$. Furthermore, this construction gives rise to a functor
\[\widetilde{C}^{\Delta}_*: \mathsf{sSet} \to \mathsf{cCoalg}_{\mathbf{k}}.\]
\end{definition}
The following result establishes a connection between the cobar functor from categorical coalgebras to dg categories and the dg nerve functor originally defined in \cite{higheralgebra}.

\begin{theorem} \label{dgnervecobar} The composition of functors
\[ \mathbf{\Omega} \circ \widetilde{C}^{\Delta}_* \colon\mathsf{sSet} \to \mathsf{dgCat}_{\mathbf{k}}\]
is naturally isomorphic to the (1-categorical) left adjoint of the dg nerve functor 
\[ \emph{N}_{dg} \colon \mathsf{dgCat}_{\mathbf{k}} \to \mathsf{sSet}.\]
\end{theorem}
\begin{proof}
This result was proved in section 4 of \cite{holstein-lazarev}. The one object case was proved in \cite{rivera2018cubical}. 
\end{proof}

The dg nerve and its left adjoint also fit into a Quillen adjunction as we now record. Denote the left adjoint of $\text{N}_{dg}$ by 
\[ \mathbf{\Lambda} \colon \mathsf{sSet} \to \mathsf{dgCat}_{\mathbf{k}}.\]

\begin{theorem} \label{dgnervequillen} The adjunction
\[ \mathbf{\Lambda} \colon \mathsf{sSet} \rightleftarrows \mathsf{dgCat}_{\mathbf{k}} \colon \emph{N}_{dg} \]
is a Quillen adjunction  of model categories when $\mathsf{sSet}$ is equipped with Joyal's model structure and $\mathsf{dgCat}_{\mathbf{k}}$ with Tabuada's model structure. 
In particular, $\mathbf{\Lambda} \cong \mathbf{\Omega} \circ \widetilde{C}^{\Delta}_*$ sends categorical equivalences of simplicial sets to quasi-equivalences of dg categories.
\end{theorem}
\begin{proof} This is Proposition 1.3.1.20 in \cite{higheralgebra}. The second statement follows since all simiplicial sets are cofibrant in Joyal's model structure.
\end{proof}

\begin{remark} In general, the functor $\mathbf{\Lambda}$ does not send weak homotopy equivalences of simplicial sets to quasi-equivalences of dg categories. However, if $f \colon X \to X'$ is a weak homotopy equivalence and $X$ and $X'$ are ``group-like", namely, their homotopy categories are groupoids, then $\mathbf{\Lambda}(f)$ is a quasi-equivalence of dg categories. 
\end{remark}

\subsection{Chains as a $B_{\infty}$-categorical coalgebra.} \label{chainsccoalg}
We now describe a natural lift of $\widetilde{C}^{\Delta}_*$ to the category $\mathsf{B}_{\infty}\text{-}\mathsf{cCoalg}_{\mathbf{k}}.$ For simplicity we will denote this lift by 
\[ C_* \colon \mathsf{sSet} \to \mathsf{B}_{\infty}\text{-}\mathsf{cCoalg}_{\mathbf{k}}.\]

To define this lift we use the factorization of $\mathbf{\Lambda}$, constructed in Section 6 of \cite{rivera2018cubical}, through the category of small categories enriched over cubical sets (with connections).  More precisely, in \cite{rivera2018cubical} we constructed a functor
\[ \mathfrak{C}_{\square_c} \colon \mathsf{sSet} \to \mathsf{Cat}_{\square_c} \]
from simplicial sets to the category of small categories enriched over the monoidal category of cubical sets (with connections) with Day convolution product. Then we showed that $\mathbf{\Lambda}$ is naturally isomorphic to the composition 
\[  \mathsf{sSet} \xrightarrow{\mathfrak{C}_{\square_c}} \mathsf{Cat}_{\square_c} \xrightarrow{\mathfrak{Q}_{\square_c}} \mathsf{dgCat}_{\mathbf{k}},\]
where $\mathfrak{Q}_{\square_c}$ is the functor obtained by applying the monoidal functor $Q_*$ of normalized cubical chains at the level of morphisms.  

The chain complex of normalized cubical chains $Q_*(K)$ on a cubical set $K$, with or without the extra data of connections,  has a natural coproduct structure
\[\nabla \colon Q_*(K) \to Q_*(K) \otimes Q_*(K) \]
making $Q_*(K)$ into a dg coalgebra. It is completely determined by its action on the standard $0$-cube and $1$-cube given by the formulas
\[ \nabla [0] = [0] \otimes [0]\]
and
\[ \nabla \colon [0,1] \mapsto [0] \otimes [0,1] + [0,1] \otimes [1],\] respectively. 

Hence, the natural isomorphism $ \Lambda \cong \mathfrak{Q}_{\square_c} \circ \mathfrak{C}_{\square_c}$ together with this cubical coproduct provides a natural lift of $\mathbf{\Lambda} \colon \mathsf{sSet} \to \mathsf{dgCat}_{\mathbf{k}}$ to the category of small categories enriched over the monoidal category of dg $\mathbf{k}$-coalgebras. Using the identification $\mathbf{\Lambda} \cong \mathbf{\Omega} \circ \widetilde{C}^{\Delta}_*$, we may interpret this additional structure as a functor
\[ C_* \colon \mathsf{sSet} \to \mathsf{B}_{\infty}\text{-}\mathsf{cCoalg}_{\mathbf{k}}\]
lifting
\[ \widetilde{C}^{\Delta}_*\colon \mathsf{sSet} \to \mathsf{cCoalg}_{\mathbf{k}}.\]
In particular, if $X$ is a simplicial set with one vertex (i.e. a reduced simplicial set) then this construction provides the dg algebra $\mathbf{\Omega}(C^{\Delta}_*(X))$ with a natural dg bialgebra structure. 

In the one vertex case, a version of this construction has been studied in detail in \cite{baues1980geometry}, \cite{Bau81} and, more recently, in \cite{RM}. In the many vertex case, this version of the chains functor has been also discussed in \cite{MRZ}. We refer the reader to these references for more details. 

\subsection{The extended cobar construction as a model for the path category}
For any topological space $Y$ denote by $\mathcal{P}Y$ the topologically enriched category whose objects are the points of $Y$ and morphisms $\mathcal{P}Y(x,y)$ are given by the space (with compact-open topology) of pairs $(r, \gamma)$ where $r$ is a non-negative real number (which we call the ``parameter") and $\gamma \colon [0,r] \to Y$ a continuous path with $\gamma(0)=x$ and $\gamma(r)=y$. Composition is given by concatenation of paths and adding the corresponding parameters. Identities are constant paths with parameter $r=0$.
We call $\mathcal{P}Y$ the \textit{path category} of $Y$.

Denote by $C_*^{\text{sing}}(\mathcal{P}Y)$ the dg $\mathbf{k}$-category obtained by applying the normalized singular chains functor (equipped with the Elienberg-Zilber lax structure) on the morphisms of the topologically enriched category $\mathcal{P}Y$. This gives rise to a functorial construction \[ \mathsf{Top} \to \mathsf{dgCat}_{\mathbf{k}}\] \[Y \mapsto C_*^{\text{sing}}(\mathcal{P}Y)\] that sends weak homotopy equivalences of spaces to quasi-equivalence of dg categories. 

\begin{theorem} \label{pathcategory} For any simplicial set $X$, the dg categories $\widehat{\mathbf{\Omega}}(C_*(X))$ and $C_*^{\emph{sing}}(\mathcal{P}|X|)$ are naturally quasi-equivalent.
\end{theorem}
\begin{proof}
The natural map 
\[  \mathbf{k}[ \mathcal{Z}(C_*(X))] \to \mathbf{\Omega}(C_*(X))\]
is a cofibration between cofibrant dg categories, since, by Theorem \ref{dgnervecobar}, it may be identified with
\[ \mathbf{\Lambda}(i) \colon \mathbf{\Lambda}(\text{sk}_1X) \to \mathbf{\Lambda}(X),\]
where $\mathbf{\Lambda} \colon \mathsf{sSet} \to \mathsf{dgCat}_{\mathbf{k}}$ is the left adjoint of the dg nerve functor and $i \colon \text{sk}_1X \hookrightarrow X$ the inclusion of the $1$-skeleton of $X$ into $X$. By Remark \ref{derivedloc}, the (ordinary) pushout $\widehat{\mathbf{\Omega}}(C_*(X)) \cong \mathbf{\Lambda}(X)[X_1^{-1}]$ is actually a homotopy pushout of the maps $\mathbf{\Lambda}(i) \colon \mathbf{\Lambda}(\text{sk}_1X) \to \mathbf{\Lambda}(X)$ and $\mathbf{\Lambda}(\text{sk}_1X) \to \mathbf{\Lambda}(\text{sk}_1X)[X_1^{-1}]$, where $\mathbf{\Lambda}(\text{sk}_1X)[X_1^{-1}]$ is the dg category obtained by linearizing the free groupoid generated by the quiver $X_1 \rightrightarrows X_0$ determined by the first two simplicial face maps. Let \[\mathcal{K} \colon \mathsf{sSet} \to \mathsf{sSet}\] be a Kan replacement functor so that there is a natural quasi-equivalence of dg categories \[\mathbf{\Lambda}(\text{sk}_1X)[X_1^{-1}] \simeq \mathbf{\Lambda}(\mathcal{K}(\text{sk}_1X)).\]

By Theorem \ref{dgnervequillen}, $\mathbf{\Lambda} \colon \mathsf{sSet} \to \mathsf{dgCat}_{\mathbf{k}}$ is a left Quillen functor between Joyal's model structure on simplicial sets and Tabuada's model structure on dg categories, thus $\mathbf{\Lambda}$ preserves homotopy pushouts. 
Hence, we have natural quasi-equivalence of dg categories
\[\widehat{\mathbf{\Omega}}(C_*(X))\cong \mathbf{\Lambda}(X)[X_1^{-1}] \simeq \mathbf{\Lambda}( X \coprod_{\text{sk}_1X} \mathcal{K}(\text{sk}_1X)).\]  Note the map \[ X \to X \coprod_{\text{sk}_1X} \mathcal{K}(\text{sk}_1X)\] is a weak homotopy equivalence of simplicial sets and the homotopy category of $X \coprod_{\text{sk}_1X} \mathcal{K}(\text{sk}_1X)$ is a groupoid. It follows that $\mathbf{\Lambda}( X \coprod_{\text{sk}_1X} \mathcal{K}(\text{sk}_1X))$ is naturally quasi-equivalent to the dg category $C_*^{\text{sing}}(\mathcal{P}|X \coprod_{\text{sk}_1X} \mathcal{K}(\text{sk}_1X)|)$ and consequently to $C_*^{\text{sing}}(\mathcal{P}|X|)$, as desired. 

\end{proof}

\section{The coHochschild complex of a categorical coalgebra}
We define a version of the coHochschild complex for categorical coalgebras. Then we establish a relationship with the Hochschild complex of a dg category. The coHochschild complex in the case of connected dg coalgebras has been sutided in \cite{hesscohochschild} and \cite{hess-shipley}.
\subsection{The coHochschild complex}
We construct a functor
\[\mathcal{coCH}_* \colon \mathsf{cCoalg}_{\mathbf{k}} \to \mathsf{Ch}_{\mathbf{k}}^{\geq 0},
\]
called the \textit{coHochschild complex}, as follows. For any categorical coalgebra $C$, the underlying graded $\mathbf{k}$-module of $\mathcal{coCH}_*(C)$ is defined by
\[ C \underset{C_0 \otimes C_0^{op}}{\msquare} \mathcal{M}(\mathbf{\Omega}(C)):= (C \underset{C_0}{\msquare} \mathcal{M}(\mathbf{\Omega}(C))) \bigcap (C \underset{C_0^{op}}{\msquare} \mathcal{M}(\mathbf{\Omega}(C)))
\]
Explicitly, this notation is saying that $\mathcal{coCH}_k(C)$ is generated as a graded $\mathbf{k}$-module by monomials
\[x=x_0 \square s^{-1}x_1 \square \cdots \square s^{-1}x_p=x_0\{x_1| \cdots |x_p\},
\]
where $x_0 \in C$, $x_i \in \overline{C}$ for $i=1,\cdots,p$, $x_p \otimes x_0 \in C \underset{C_0}{\Box} C$, and $|x_0| + |x_1|+ \cdots + |x_p| - p = k$. 
The differential 
\[\partial \colon \mathcal{coCH}_k(C) \to \mathcal{coCH}_{k-1}(C)\] is defined by
\begin{eqnarray*} \partial(x)= d_Cx_0 \{ x_1 |\cdots | x_p \}
+\sum_{i=1}^p \pm x_0 \{ x_1 | \cdots|d_Cx_i | \cdots | x_p\}+ \sum_{i=1}^{p} \pm x_0 \{x_1 | \cdots | \overline{h}(x_i) | \cdots |x_p\}
\\
+ \sum_{(x_0)} \pm x_0' \{ x_0''| x_1 | \cdots | x_p\}
+ \sum_{i=1}^p \sum_{(x_i)} \pm x_0 \{ x_1 | \cdots |x_i' | x_i''|  \cdots  |x_p\}
+ \sum_{(x_0)} \pm x_0'' \{ x_1 | \cdots| x_p |x_0'\}. 
\end{eqnarray*} 
The signs are determined, as usual, from the Koszul sign convention. One may also equip the chain complex $\mathcal{coCH}_*(C)$ with the further structure of a mixed complex by defining a degree $+1$ operator
\[P \colon \mathcal{coCH}_*(C) \to \mathcal{coCH}_{*+1}(C)\] through the formula
\begin{eqnarray*}
P(x_0 \{x_1 | \cdots | x_p\})=
\\
\sum_{i=1}^p \pm \varepsilon(x_0)x_i \{ x_{i+1} | \cdots |  x_p |  x_1|  \cdots |x_{i-1}\},
\end{eqnarray*}
where $\varepsilon: C\to \mathbf{k}$ denotes the counit of $C$.
A straightforward computation yields that \[ (\mathcal{coCH}_*(C), \partial, P)\] is a non-negatively graded mixed $\mathbf{k}$-complex. This construction is clearly functorial with respect to maps of categorical coalgebras. 

\subsection{The extended coHochschild complex} Define the \textit{extended Hochschild complex} as the functor
\[\widehat{\mathcal{coCH}}_* \colon \mathsf{B}_{\infty}\text{-}\mathsf{cCoalg}_{\mathbf{k}} \to \mathsf{Ch}_{\mathbf{k}}^{\geq 0}
\]
given by
\[ C \underset{C_0 \otimes C_0^{op}}{\msquare} \mathcal{M}(\widehat{\mathbf{\Omega}}(C)).
\]
%Note $\widehat{\mathcal{C}}_*(C)$ is now generated by monomials 
%\[x_0 \square s^{-1} x_1 \square \cdots \square s^{-1} x_p\]
%where $x_0 \in C$,  $s^{-1} x_1 \square \cdots \square s^{-1} x_p \in \mathcal{M}(\widehat{\mathbf{\Omega}}(C))$, and $x_p \otimes x_0 \in C \square_{C_0}C.$
\subsection{Relationship with the Hochschild chain complex}
%When $\mathsf{A}$ is a dg category obtained by applying the cobar construction of \ref{cobarconstruction} to a categorical coalgebra, the two sided bar construction $\text{Bar}(\mathsf{M}, \mathsf{A}, \mathsf{N})$ defined in \ref{barconstruction} may be simplified as follows. 

Let $C$ be a categorical coalgebra and $\mathsf{M}$ and $\mathsf{N}$ right and left dg modules over $\mathcal{M}(\mathbf{\Omega}(C))$, respectively, as in \ref{barconstruction}. Define a graded $\mathbf{k}$-module
 \[\mathcal{Q}(\mathsf{M}, C, \mathsf{N}):= \mathsf{M} \underset{C_0}{\square} C \underset{C_0}{\square}  \mathsf{N} \]
and consider the linear map \[ \partial_{\mathcal{Q}}  \colon \mathcal{Q}(\mathsf{M}, C, \mathsf{N}) \to \mathcal{Q}(\mathsf{M}, C, \mathsf{N})\]
of degree $-1$ defined by
\[
\partial_{\mathcal{Q}}=
d_{\mathsf{M}} \square \text{id}_{C} \square \text{id}_{\mathsf{N}} + \text{id}_{\mathsf{M}} \square d_C \square \text{id}_{\mathsf{N}} + \text{id}_{\mathsf{M}}  \square \text{id}_{C} \square d_{\mathsf{N}} + \theta', \]
where
\[\theta'(m \square c \square n)=
\pm ( m \cdot \{c'\} ) \square c'' \square n \pm m \square  c' \square (\{c''\} \cdot n).\]
\begin{remark}
Note that $\partial_{\mathcal{Q}}$ may not square to zero, since $d_C$ may not square to zero in a categorical coalgebra. When $C$ is a connected dg coalgebra then $\mathcal{Q}( \mathsf{M}, C, \mathsf{N})$ is indeed a chain complex. 

\end{remark}
We will now define a map $H$ of degree $+1$ and two maps $\pi$ and $\alpha$ of degree $0$ fitting in the diagram
\[
\begin{tikzcd}
	 {\text{Bar}_{C_0}(\mathsf{M}, \mathbf{\Omega}(C), \mathsf{N})}\arrow[loop left, distance=3em, start anchor={[yshift=-1ex]west}, end anchor={[yshift=1ex]west}]{}{H}& {\mathcal{Q}(\mathsf{M}, C, \mathsf{N}).}
	\arrow["\pi", shift left=2, from=1-1, to=1-2]
	\arrow["\alpha", shift left=2, from=1-2, to=1-1]
\end{tikzcd}
\]
These maps will satisfy three equations given in Proposition \ref{contraction}.

\begin{enumerate}
\item Define \[\pi\colon \text{Bar}_{C_0}(\mathsf{M}, \mathbf{\Omega}(C), \mathsf{N}) \to \mathcal{Q}(\mathsf{M}, C, \mathsf{N})\]
on any generator $m[a_1 \big| \cdots \big| a_p]n$ by letting
\[ \pi( m [a_1 \big| \cdots \big| a_p]  n)= 0 \text{ if $p>1$ }, \]
and when $p=1$, writing $a_1=\{c_1| \cdots |c_q\}$, define
\[ \pi(m [ \{c_1 | \cdots |c_q\}])n = \sum_{i=1}^q m \cdot \{c_1| \cdots |c_{i-1}\} \square [\{c_i\}] \square \{c_{i+1} | \cdots | c_q\} \cdot n\] if $q>0$, and
\[ \pi(m [ x ] n) = m \square x \square n\] if $q=0$ and $ x \in \mathcal{S}(C) \subset C_0.$
\item Define 
\[ \alpha \colon \mathcal{Q}(\mathsf{M}, C, \mathsf{N}) \to \text{Bar}_{C_0}(\mathsf{M}, \mathbf{\Omega}(C), \mathsf{N}) \]
by
\[ \alpha(m \square c \square n)= m [\{ c\} ] n + \sum m [\{c'\} \big| \{c''\}]n + \sum m [\{c'\} \big| \{c''\} \big| \{c'''\}]n + \cdots \]
Note $\alpha$ is well defined since the induced coproduct $\Delta \colon \overline{C} \to \overline{C} \otimes \overline{C}$ is of degree $0$ and $\overline{C}$ is concentrated on positive degrees. 
\item Define \[H:\text{Bar}_{C_0}(\mathsf{M}, \mathbf{\Omega}(C), \mathsf{N})  \to \text{Bar}_{C_0}(\mathsf{M}, \mathbf{\Omega}(C), \mathsf{N}) \] to be a degree $+1$ linear map given on a generator $m[a_1 \big| \cdots \big| a_p]n$ as follows. Write $a_1=\{c_1|\cdots|c_m\}$ and let
\[ H( m[\{c_1|\cdots|c_m\} \big| a_2 \big| \cdots \big| a_p]n) =0  \text{ if $m<2$ } ,\]
\begin{eqnarray*}
H(m[ \{c_1|c_2\} \big| a_2 \big| \cdots \big| a_p]n)= m[\{c_1\} \big| \{c_2\} \big| a_2 \big| \cdots \big| a_p]n + \\
\sum [\{c_1'\} \big| \{c_1''\} \big| \{c_2\} \big| a_2 \big| \cdots \big| a_p]n + \\
\sum [\{c_1'\} \big| \{c_1''\} \big|\{c_1'''\}\big| \{c_2\} \big| a_2 \big| \cdots \big| a_p]n + \cdots \text{ if $m=2$, }
\end{eqnarray*}
and, if $m >2$, let
\begin{eqnarray*}
H( m[\{c_1|\cdots|c_m\} \big| a_2 \big| \cdots \big| a_p]n) =
\\
  \sum_{i=1}^{m}m \cdot \{c_1| \cdots |c_{i-1}\} [ \{c_i \} \big| \{c_{i+1}| \cdots | c_m\} \big| a_2 \big| \cdots \big| a_p]n +
\\
\sum_{i=1}^{m}m \cdot \{c_1| \cdots |c_{i-1}\} [ \{c_i' \} \big| \{c_i''\} \big| \{c_{i+1}| \cdots | c_m\} \big| a_2 \big| \cdots \big| a_p]n +
\\
\sum_{i=1}^{m}m \cdot \{c_1| \cdots |c_{i-1}\} [ \{c_i' \} \big| \{c_i''\} \big| \{c_i'''\} \big| \{c_{i+1}| \cdots | c_m\} \big| a_2 \big| \cdots \big| a_p]n + \cdots
\end{eqnarray*}
\end{enumerate}
A tedious but straightforward computation verifies the following equations hold.
\begin{proposition} \label{contraction}
The maps $\pi$, $\alpha$, and $H$ defined above satisfy the equations 
 \begin{eqnarray}
\pi \circ \partial_{\mathsf{M}, C, \mathsf{N}}= \partial_{\mathcal{Q}} \circ \pi,
\\
\alpha \circ \partial_{\mathcal{Q}} = \partial_{\mathsf{M}, C, \mathsf{N}} \circ \alpha, \emph{  and  }
\\
H \circ \partial_{\mathsf{M}, C, \mathsf{N}} +\partial_{\mathsf{M}, C, \mathsf{N}} \circ H=  \alpha \circ \pi - \emph{id}_{\emph{Bar}_{C_0}(\mathsf{M}, \mathbf{\Omega}(C), \mathsf{N})}.
\end{eqnarray}
\end{proposition}

We avoid using the terminology ``chain contraction" in the above proposition precisely because $\partial_{\mathcal{Q}}$ might not square to zero. 
In any case, as our main application we consider the case $\mathsf{M}=\mathsf{N}=\mathcal{M}(\mathbf{\Omega}(C))$. In this particular case, 
\[\partial_{\mathcal{Q}} \colon \mathcal{Q}(\mathcal{M}(\mathbf{\Omega}(C)), C, \mathcal{M}(\mathbf{\Omega}(C))) \to  \mathcal{Q}(\mathcal{M}(\mathbf{\Omega}(C)), C, \mathcal{M}(\mathbf{\Omega}(C))) \]
does square to zero and so $(\mathcal{Q}(\mathcal{M}(\mathbf{\Omega}(C)), C, \mathcal{M}(\mathbf{\Omega}(C))) , \partial_{\mathcal{Q}})$ defines a dg $\mathbf{k}$-module. In fact, it follows from the definition of a categorical coalgebra that the two terms  \[ \pm m \square d_C^2 (c)\square n \]
and
\[ \pm ( m \cdot \overline{h}\{c'\} ) \square c'' \square n \pm m \square  c' \square (\overline{h}\{c''\} \cdot n)  \]
in $\partial_{\mathcal{Q}}^2(m \square c \square n)$ cancel each other.

\begin{theorem}\label{hoch-cohoch} For any categorical coalgebra $C$, there is a natural chain contraction of dg $\mathbf{k}$-modules
\[
\begin{tikzcd}
	 {\mathcal{CH}_*(\mathbf{\Omega}(C)) }\arrow[loop left, distance=3em, start anchor={[yshift=-1ex]west}, end anchor={[yshift=1ex]west}]{}{\overline{H}}& {\mathcal{coCH}_*(C).}
	\arrow["\overline{\pi}", shift left=2, from=1-1, to=1-2]
	\arrow["\overline{\alpha}", shift left=2, from=1-2, to=1-1]
\end{tikzcd}
\]
If $C$ is a $B_{\infty}$-categorical coalgebra, then there is a natural chain contraction of dg $\mathbf{k}$-modules
\[
\begin{tikzcd}
	 {\emph{Bar}_{C_0}(\mathcal{M}(\widehat{\mathbf{\Omega}}(C)), \mathbf{\Omega}(C),\mathcal{M}(\widehat{\mathbf{\Omega}}(C))) \underset{\mathcal{M}(\widehat{\mathbf{\Omega}}(C)) \otimes \mathcal{M}(\widehat{\mathbf{\Omega}}(C))^{op}}{\bigotimes} \widehat{\mathbf{\Omega}}(C)} \arrow[loop left, distance=3em, start anchor={[yshift=-1ex]west}, end anchor={[yshift=1ex]west}]{}{\widehat{H}}& {\widehat{\mathcal{coCH}}_*(C).}
	\arrow["\widehat{\pi}", shift left=2, from=1-1, to=1-2]
	\arrow["\widehat{\alpha}", shift left=2, from=1-2, to=1-1]
\end{tikzcd}
\]
\end{theorem}
\begin{proof}
Recall \[ \mathcal{CH}_*(\mathbf{\Omega}(C))= \text{Bar}_{C_0}(\mathcal{M}(\mathbf{\Omega}(C)), \mathbf{\Omega}(C), \mathcal{M}(\mathbf{\Omega}(C))) 
\underset{\mathcal{M}(\mathbf{\Omega}(C)) \otimes \mathcal{M}(\mathbf{\Omega}(C))^{op}}{\bigotimes} 
\mathcal{M}(\mathbf{\Omega}(C)) .\]
Now we observe there is a natural isomorphism of dg $\mathbf{k}$-modules
\[ \mathcal{coCH}_*(C) \cong \mathcal{Q}(\mathcal{M}(\mathbf{\Omega}(C)), C, \mathcal{M}(\mathbf{\Omega}(C))) \underset{\mathcal{M}(\mathbf{\Omega}(C)) \otimes \mathcal{M}(\mathbf{\Omega}(C))^{op}}{\bigotimes} \mathcal{M}(\mathbf{\Omega}(C)).
\]
%The differential $\partial \colon \mathcal{coCH}_*(C) \to \mathcal{coCH}_{*-1}(C) $ of the coHochschild complex corresponds to the map \[\partial_{\mathcal{Q}} \underset{\mathcal{M}(\mathbf{\Omega}(C)) \otimes \mathcal{M}(\mathbf{\Omega}(C))^{op}}{\bigotimes} \text{id}_{\mathcal{M}(\mathbf{\Omega}(C))} + \text{id}_{\mathcal{Q}(\mathsf{M}, C, \mathsf{N})}  \underset{\mathcal{M}(\mathbf{\Omega}(C)) \otimes \mathcal{M}(\mathbf{\Omega}(C))^{op}}{\bigotimes} D_{\mathcal{M}(\mathbf{\Omega}(C))} \] under this isomorphism. 

Using the notation of Proposition \ref{contraction}, define
\[ \overline{\pi}= \pi \underset{\mathcal{M}(\mathbf{\Omega}(C)) \otimes \mathcal{M}(\mathbf{\Omega}(C))^{op}}{\bigotimes} \text{id}_{\mathcal{M}(\mathbf{\Omega}(C))}, \]
\[\overline{\alpha}= \alpha \underset{\mathcal{M}(\mathbf{\Omega}(C)) \otimes \mathcal{M}(\mathbf{\Omega}(C))^{op}}{\bigotimes} \text{id}_{\mathcal{M}(\mathbf{\Omega}(C))}, \text{ and }\]
\[\overline{H}= H \underset{\mathcal{M}(\mathbf{\Omega}(C)) \otimes \mathcal{M}(\mathbf{\Omega}(C))^{op}}{\bigotimes} \text{id}_{\mathcal{M}(\mathbf{\Omega}(C))}.\]
It follows from Proposition \ref{contraction} that  $\overline{\pi}$, $\overline{\alpha}$ and $\overline{H}$ define the data of a natural chain contraction. The proof of the second statement is similar.
\end{proof}

\subsection{Invariance of the coHochschild complex} 
Recall a morphism of dg categories $f \colon \mathsf{A} \to \mathsf{A}'$ is a \textit{quasi-equivalence} if
\begin{enumerate}
\item for any two objects $x$ and $y$ in $\mathsf{A}$, the induced map $f_{x,y} \colon \mathsf{A}(x,y) \to \mathsf{A}'(f(x), f(y))$ is a quasi-isomorphism of dg $\mathbf{k}$-modules, and
\item if the induced map on homotopy categories $H_0(f) \colon H_0(\mathsf{A}) \to H_0(\mathsf{A}')$ is essentially surjective.
\end{enumerate}
\begin{definition}
    A map $f\colon C \to C'$ between categorical coalgebras is called an $\mathbf{\Omega}$-\textit{quasi-equivalence} if $\mathbf{\Omega}(f)\colon \mathbf{\Omega}(C) \to \mathbf{\Omega}(C')$ is a quasi-equivalence of dg categories. 
\end{definition}
The coHochschild complex is invariant under $\mathbf{\Omega}$-quasi-equivalences as shown next.
\begin{proposition}
If $f \colon C \to C' $ is an $\mathbf{\Omega}$-quasi-equivalence between categorical coalgebras $C$ and $C'$, then
\[\mathcal{coCH}_*(f) \colon \mathcal{coCH}_*(C)\to \mathcal{coCH}_*(C')\] 
is a quasi-isomorphism of dg $\mathbf{k}$-modules. 
\end{proposition}
\begin{proof}
Since $C$ and $C'$ are flat as $\mathbf{k}$-modules, the dg categories $\mathbf{\Omega}(C)$ and $\mathbf{\Omega}(C')$ are locally $\mathbf{k}$-flat.  
Hence, the quasi-equivalence
\[\mathbf{\Omega}(f)\colon \mathbf{\Omega}(C) \to \mathbf{\Omega}(C')\] induces a quasi-isomorphism between Hochschild complexes
\[ \mathcal{CH}_*(\mathbf{\Omega}(f)) \colon \mathcal{CH}_*(\mathbf{\Omega}(C)) \to \mathcal{CH}_*(\mathbf{\Omega}(C')).
\]By Corollary \ref{hoch-cohoch}, we have natural quasi-isomorphisms
\[\mathcal{CH}_*(\mathbf{\Omega}(C))\simeq \mathcal{coCH}_*(C) \]
and 
\[\mathcal{CH}_*(\mathbf{\Omega}(C'))\simeq \mathcal{coCH}_*(C').\]
It follows that the induced map
\[\mathcal{coCH}_*(f) \colon \mathcal{coCH}_*(C) \to \mathcal{coCH}_*(C')\] is a quasi-isomorphism. 
\end{proof}

\section{The coHochschild complex as a model for the free loop space}
We establish a relationship between the extended coHochschild complex and the free loop space. For any topological space $Y$, we denote by $LY$ the \textit{free loop space} of $Y$ modeled as the space (with compact-open topology) of pairs $(r,\gamma)$ where $r$ is a non-negative real number and $\gamma \colon [0,r] \to Y$ a continous map with $\gamma(0)=\gamma(r)$. Denote by $C_*^{\text{sing}}(LY)$  the dg $\mathbf{k}$-module of normalized singular chains on $LY$. Let $\mathbf{k}$ be a commutative ring and for any simplicial set $X$ denote by $C_*(X)$ its $B_{\infty}$-categorical $\mathbf{k}$-coalgebra of chains. Note that the $S^1$-action on $L|X|$ given by rotating loops gives rise to an operator $R: C_*^{\text{sing}}(L|X|) \to C_{*+1}^{\text{sing}}(L|X|)$ which gives the chain complex $C_*^{\text{sing}}(L|X|)$ the extra structure of a mixed complex.

\begin{theorem} For any simplicial set $X$, the dg $\mathbf{k}$-modules
$\widehat{\mathcal{coCH}}_*(C_*(X))$ and $C_*^{\emph{sing}}(L|X|)$ are naturally quasi-isomorphic. 
\end{theorem}
\begin{proof}
For simplicity write $C=C_*(X)$. By Theorem \ref{hoch-cohoch}, there is a natural chain homotopy equivalence
\[\widehat{\mathcal{coCH}}_*(C) \simeq \text{Bar}_{C_0}(\mathcal{M}(\widehat{\mathbf{\Omega}}(C)), \mathbf{\Omega}(C),\mathcal{M}(\widehat{\mathbf{\Omega}}(C))) \underset{\mathcal{M}(\widehat{\mathbf{\Omega}}(C)) \otimes \mathcal{M}(\widehat{\mathbf{\Omega}}(C))^{op}}{\bigotimes} \widehat{\mathbf{\Omega}}(C).\]
Note that $\mathbf{\Omega}(C)$ is a cofibrant dg category being naturally isomorphic to $\mathbf{\Lambda}(X)$ where $\mathbf{\Lambda} \colon \mathsf{sSet} \to \mathsf{dgCat}_{\mathbf{k}}$ denotes the left adjoint of the dg nerve functor. Hence, the natural map \[\mathbf{\Omega}(C) \cong \mathbf{\Lambda}(X) \to \mathbf{\Lambda}(X)[X_1^{-1}] \cong \widehat{\mathbf{\Omega}}(C)\] induces a quasi-isomorphism after applying bar constructions. This is a classical fact, for instance see Section Q5 in \cite{QuillenAp}. More precisely, the natural map
\[\text{Bar}_{C_0}(\mathcal{M}(\widehat{\mathbf{\Omega}}(C)), \mathbf{\Omega}(C),\mathcal{M}(\widehat{\mathbf{\Omega}}(C))) \to \text{Bar}_{C_0}(\mathcal{M}(\widehat{\mathbf{\Omega}}(C)), \widehat{\mathbf{\Omega}}(C),\mathcal{M}(\widehat{\mathbf{\Omega}}(C)))\] is a quasi-isomorphism.
Consequently, we have a natural quasi-isomorphism 
\[ \widehat{\mathcal{coCH}}_*(C) \xrightarrow{\simeq}
\text{Bar}_{C_0}(\mathcal{M}(\widehat{\mathbf{\Omega}}(C)), \widehat{\mathbf{\Omega}}(C),\mathcal{M}(\widehat{\mathbf{\Omega}}(C))) \underset{\mathcal{M}(\widehat{\mathbf{\Omega}}(C)) \otimes \mathcal{M}(\widehat{\mathbf{\Omega}}(C))^{op}}{\bigotimes} \widehat{\mathbf{\Omega}}(C)\] 
\[ = \mathcal{CH}_*(\widehat{\mathbf{\Omega}}(C)).\]

The invariance of the Hochschild complex with respect to quasi-equivalences between locally $\mathbf{k}$-flat dg categories, together with Theorem \ref{pathcategory}, implies that $\mathcal{CH}_*(\widehat{\mathbf{\Omega}}(C))$ is naturally quasi-isomorphic to  $\mathcal{CH}_*(C_*^{\text{sing}}(\mathcal{P}|X|))$. By a result proved by Goodwillie  \cite{Goodwillie} and also (independently) by Burghelea and Fiedorowicz \cite{BF}, there is a natural quasi-isomorphism \[\mathcal{CH}_*(C_*^{\text{sing}}(\mathcal{P}|X|))\xrightarrow{\simeq} C_*^{\text{sing}}(L|X|),\] as desired. 
\end{proof}

\begin{remark}
If the simplicial set $X$ has the property of being ``group-like", namely, if the homotopy category of $X$ is a groupoid, then the (non-extended) coHochschild complex $\mathcal{coCH}_*(C_*(X))$ of the underlying categorical coalgebra of $C_*(X)$ is already naturally quasi-isomorphic to $C^{\text{sing}}_*(L|X|)$. In other words, there is no need to localize if every $1$-simplex in $X$ is invertible up to homotopy. In this case,
$\mathcal{coCH}_*(C_*(X))$ and $C^{\text{sing}}_*(L|X|)$ are quasi-isomorphic as mixed complexes. This follows since the quasi-isomorphism $\overline{\pi}$ in Theorem \ref{hoch-cohoch} is a morphism of mixed complexes (i.e. it intertwines the operators $B$ and $P$) together with the fact that the quasi-isomorphism 
\[\mathcal{CH}_*(C_*^{\text{sing}}(\mathcal{P}|X|))\xrightarrow{\simeq} C_*^{\text{sing}}(L|X|)\] constructed in \cite{Goodwillie} and \cite{BF} is also a morphism of mixed complexes (i.e. intertwines the operators $B$ and $R$). A theory of cyclic homology for categorical coalgebras will be developed by Daniel Tolosa in his PhD thesis. 
\end{remark}

\section{dg Hopf algebras and the adjoint action}
This section has two goals. The first goal is to clarify the relationship between the coHochschild complex model for the free loop space and Brown's twisted tensor product model for the total space of a fibration. The latter uses the conjugation action of (a topological group model of) the based loop space on itself when defining the twisted differential, while the first does not use any antipode map (or inverses) when defining the coHochschild complex differential.

The second goal is to describe a natural chain map
\[C \to \widehat{\mathcal{coCH}}_*(C)\]
modeling the continuous map $Y \to LY$ sending each point of a space $Y$ to its corresponding constant loop in the free loop space. We expect this to be useful in string topology of non-simply connected manifolds, where constant loops play a delicate role. 

For simplicity, in this section we work with connected dg $\mathbf{k}$-coalgebras instead of categorical $\mathbf{k}$-coalgebras, namely, we work in the one object case. For homotopy theoretic applications this is not a strong hypothesis since pointed connected homotopy types may be modeled by reduced simplicial sets (i.e. simplicial sets with a single vertex). For any reduced simplicial set $X$, we denote by $C_*(X)$ the normalized simplicial chains on $X$, which is a connected dg coalgebra when equipped with the Alexander-Whitney coproduct. It follows from Section \ref{chainsccoalg} that the dg algebra (or dg category with one object) $\widehat{\mathbf{\Omega}}(C_*(X))$ has a natural dg bialgebra structure naturally quasi-isomorphic to the dg bialgebra structure on the singular chain complex $C^{\text{sing}}_*(\Omega_b|X|)$ of based (Moore) loops in $|X|$ at $b$. One of the main technical steps in this section is showing that the dg bialgebra $\widehat{\mathbf{\Omega}}(C_*(X))$ has the property of being a dg Hopf algebra. This means there is a map of dg $\mathbf{k}$-modules
\[s \colon \widehat{\mathbf{\Omega}}(C_*(X)) \to \widehat{\mathbf{\Omega}}(C_*(X))\]
satisfying
\begin{eqnarray} \label{antipode} \mu \circ (s \otimes \text{id}) \circ \nabla = \eta \circ \varepsilon= \mu \circ (\text{id} \otimes s) \circ \nabla,
\end{eqnarray}
where $\nabla \colon \widehat{\mathbf{\Omega}}(C_*(X)) \to \widehat{\mathbf{\Omega}}(C_*(X)) \otimes \widehat{\mathbf{\Omega}}(C_*(X))$ is the coproduct, $\mu \colon \widehat{\mathbf{\Omega}}(C_*(X)) \otimes \widehat{\mathbf{\Omega}}(C_*(X)) \to \widehat{\mathbf{\Omega}}(C_*(X))$ the product, $\varepsilon \colon \widehat{\mathbf{\Omega}}(C_*(X)) \to \mathbf{k}$ the counit, and $\eta \colon \mathbf{k} \to \widehat{\mathbf{\Omega}}(C_*(X))$ the unit. 

\subsection{Hochschild chain complex and adjoint action}\label{Hochschildadjoint}

Given a map $f \colon A \to B$ of dg $\mathbf{k}$-algebras,  denote by $f^*B$ the left dg $A \otimes A^{op}$-module whose underlying dg $\mathbf{k}$-module is $B$ and the $A \otimes A^{op}$-right action is induced through $f$. Denote by
\[\mathcal{CH}_*(A,f^*B) := \text{Bar}(A,A,A) \otimes_{A \otimes A^{op}} f^*B \]
the (normalized) Hochschild chain complex of $A$ with coefficients in $f^*(B)$. This has \[ \bigoplus_{p=0}^{\infty} (s^{+1}\overline{A})^{\otimes p} \otimes f^*B \]as underlying graded $\mathbf{k}$-module and we write generators $[a_1\big| \cdots \big|a_p]b$ as usual.

Now suppose $f: A\to B$ is a map of dg $\mathbf{k}$-\textit{bialgebras}. Furthermore, suppose the dg bialgebra $B$ has the property of being a dg Hopf algebra. Denote by $f^*_{ad}(B)$ the left $A$-module whose underlying dg $\mathbf{k}$-module is $B$ and left $A$-action given by
\[ a \cdot b= \sum_{(a)} (-1)^{|a'|(|a''|+|b|)}f(a'')bs(f(a')),\]
for any $a\in A$, $b \in B$, where $s: B \to B$ denotes the antipode of $B$. We call this the \textit{adjoint action of $A$ on $B$ via $f$.}

The counit $\epsilon_A: A \to \mathbf{k}$ of the dg Hopf algebra $A$ makes $\mathbf{k}$ into a right dg $A$-bimodule. Define a dg $\mathbf{k}$-module
\[\mathcal{C}_*(A, f^*_{ad}B) := \text{Bar}(\mathbf{k}, A, f^*_{ad}B) .\]

\begin{proposition} \label{HochschildHopfalgebra} Let $f: A \to B$ be a dg $\mathbf{k}$-bialgebra map, where $B$ has the property of being a dg Hopf algebra. There is a natural isomorphism of dg $\mathbf{k}$-modules
\[\mathcal{CH}_*(A,f^*B) \cong \mathcal{C}_*(A, f^*_{ad}B).\]
\end{proposition} 

\begin{proof}
Define  
\begin{eqnarray}
\varphi \colon \mathcal{CH}_*(A,f^*B)  \to  \mathcal{C}_*(A, f^*_{ad}B)\\
\varphi([a_1 \big| \cdots \big| a_p]b) := \sum \pm [a_1''\big| \cdots \big|a_p'']bf(a_1' \cdots a_p').
\end{eqnarray}
The antipode compatibility (equation \ref{antipode}) in the definition of a dg Hopf algebra implies that $\varphi$ is a chain map with (strict) inverse given by the chain map
\begin{eqnarray} 
\varphi^{-1}([a_1 \big| \cdots \big| a_p] b )= \sum \pm [a_1'' \big| \cdots \big| a_p'']bs(f(a_1' \cdots a_p')),
\end{eqnarray}
where $s: B \to B$ denotes the antipode of $B$.
\end{proof}
\subsection{The cobar construction as a dg Hopf algebra}
We wish to apply the above discussion to a dg Hopf algebra model for the based loop space. Recall that, as discussed in \ref{chainsccoalg}, for any reduced simplicial set $X$, there is a natural coproduct \[\nabla \colon \widehat{\mathbf{\Omega}}(C_*(X)) \to \widehat{\mathbf{\Omega}}(C_*(X)) \otimes \widehat{\mathbf{\Omega}}(C_*(X))\]
making $(\widehat{\mathbf{\Omega}}(C_*(X)), D, \otimes, \nabla)$ into a dg bialgebra, which turns out to be a dg Hopf algebra, as shown next.
\begin{theorem} \label{cobarhopf} For any reduced simplicial set $X$, the dg bialgebra $(\widehat{\mathbf{\Omega}}(C_*(X)), D, \otimes, \nabla)$ has the property of being a dg Hopf algebra. 
\end{theorem}
\begin{proof}
We must show the existence of an antipode map 
\[s \colon \widehat{\mathbf{\Omega}}(C_*(X)) \to \widehat{\mathbf{\Omega}}(C_*(X)) \] that is also a map of dg $\mathbf{k}$-modules. First note that the $\mathbf{k}$-bialgebra $\widehat{\mathbf{\Omega}}(C_*(X))_0$ on degree $0$ is a Hopf algebra being isomorphic to the group algebra $\mathbf{k}[G(X_1)]$, where $G(X_1)$ denotes the free group generated by the set of $1$-simplices. The antipode \[ s_0 \colon \widehat{\mathbf{\Omega}}(C_*(X))_0 \to \widehat{\mathbf{\Omega}}(C_*(X))_0 \] is determined by sending a group-like element to its inverse. Recall that $\widehat{\mathbf{\Omega}}(C_*(X))$ if and only if the identity map $\text{id} \colon \widehat{\mathbf{\Omega}}(C_*(X)) \to \widehat{\mathbf{\Omega}}(C_*(X))$ is invertible as an element in the convolution algebra $(\text{Hom}(\widehat{\mathbf{\Omega}}(C_*(X)),\widehat{\mathbf{\Omega}}(C_*(X))), \star)$. 
So far we know that the restriction \[\text{id}|_{\widehat{\mathbf{\Omega}}(C_*(X))_0} \colon \widehat{\mathbf{\Omega}}(C_*(X))_0 \to \widehat{\mathbf{\Omega}}(C_*(X))\] is an invertible element in the convolution algebra. We now explain why $\text{id}$ is invertible following a classical argument of Takeuchi adapted to the dg setting. First recall that Hess and Tonks constructed in \cite{hess2010loop} a chain homotopy equivalence between $\widehat{\mathbf{\Omega}}(C_*(X))$ and $G(X)$, the simplicial group functorially associated to $X$ known as Kan's loop group construction. In particular, this involved constructing two maps
of dg algebras \[ \phi \colon \widehat{\mathbf{\Omega}}(C_*(X)) \to C_*(G(X)) \]
and
\[\psi \colon C_*(G(X)) \to \widehat{\mathbf{\Omega}}(C_*(X))\]
that are chain homotopy inverses to each other, see Theorem 15 in \cite{hess2010loop}. Since $G(X)$ is a simplicial \textit{group}, the chain complex $C_*(G(X))$ may be equipped with a natural dg Hopf algebra structure with product induced by that of $G(X)$ (together with the Eilenberg-Zilber map) and Alexander-Whitney coproduct. The antipode \[s_{G(X)} \colon C_*(G(X)) \to C_*(G(X))\] is induced by applying the normalized chains functor to the inverse map \[(-)^{-1} \colon G(X) \to G(X).\] Furthermore, the map $g \colon \widehat{\mathbf{\Omega}}(C_*(X)) \to \widehat{\mathbf{\Omega}}(C_*(X))$ defined by $g:=\psi \circ s_{G(X)} \circ \phi$ is a map of dg $\mathbf{k}$-modules extending $s_0$. This map might not be the inverse of $\text{id}$ in the convolution algebra but the proof of Lemma 14 in \cite{Takeuchi} explains how one may obtain such an inverse from $g$.
\end{proof}
\begin{remark} Lemma 14 in \cite{Takeuchi} implies that a graded bialgebra $B$, over a field, is a Hopf algebra if and only if the degree $0$ bialgebra $B_0$ is a Hopf algebra. However, we cannot apply this result right away to our context because the proof uses the existence of an arbitrary linear map extending the antipode in degree $0$, a fact we do not immediately have in the dg setting (even over a field) since the arbitrary extension might not preserve differentials. 
\end{remark}
\subsection{coHochschild complex and Brown's twisted tensor product}
Recall the following classical construction of Ed Brown.

\begin{definition} Let $(C, d_C, \Delta_C)$ be a dg coalgebra, $(A, d_A, \mu_A)$ an dg algebra, and $(M,d_M)$ a left dg $A$-module with action $\mu_M: A \otimes M \to M$. Suppose $\tau: C \to A$ is a twisting cochain, that is, a map of degree $-1$ satisfying \[d_A \circ \tau + \tau \circ d_C + \mu_A\circ (\tau \otimes \tau) \circ \Delta_C=0.\]
The \textit{twisted tensor product of $C$ and $M$ over $\tau$} is defined to be the chain complex $C \otimes_{\tau} M := (C \otimes M, d_{\tau})$ where \[ d_{\tau}=d_C\otimes \text{id}_M + \text{id}_C \otimes d_M + (\text{id}_C \otimes \mu_M) \circ (\text{id}_C\otimes \tau \otimes \text{id}_M) \circ (\Delta_C \otimes \text{id}_M)\]
\end{definition}
Brown proved that the above construction gives rise to a chain model for the total space $E$ of a fibration $F \to E \to B$ using a canonical twisting cochain $\tau \colon C_*(B) \to C_{*-1}(\Omega B)$ (following a geometric construction of F. Adams) and the holonomy chain map $\mu_{F} \colon C_*(\Omega B) \otimes C_*(F) \to C_*(F).$

We now compare the extended coHochschild complex of the chains on a reduced simplicial set to a twisted tensor product constructed via the adjoint action on the dg Hopf algebra structure of the cobar construction. The adjoint action is an algebraic model for the conjugation action of based loops and consequently encodes the holonomy of the free loop fibration $\Omega B \to L B \to B$. 

For any connected dg coalgebra $C$ denote by \[ \iota \colon C \to \mathbf{\Omega}(C)\] the twisting cochain given by \[\iota(x) = \{ x \}.\] The twisting cochain $\iota$ is called the \textit{universal twisting cochain of $C$}. 

Given any reduced simplicial set $X$, let $i \colon \mathbf{\Omega}(C_*(X)) \hookrightarrow \widehat{\mathbf{\Omega}}(C_*(X))$ the canonical inclusion map, and $i^*_{ad} \widehat{\mathbf{\Omega}}(C_*(X))$ the left dg $\mathbf{\Omega}(C_*(X))$-module determined by the action action, as defined in Section \ref{Hochschildadjoint} using the dg Hopf algebra structure of Theorem \ref{cobarhopf}.

\begin{theorem} \label{browncohoch} For any reduced simplicial set $X$ there is a natural chain homotopy equivalence of dg $\mathbf{k}$-modules
\[ C_*(X) \otimes_{\iota} i^*_{ad} \widehat{\mathbf{\Omega}}(C_*(X)) \leftrightarrows \widehat{\mathcal{coCH}}_*(C_*(X)).\]
\end{theorem}
\begin{proof}
For simplicity, we denote the connected dg coalgebra $C_*(X)$ by $C$, the augmented dg algebra $\mathbf{\Omega}(C_*(X))$ by $A$, and the dg Hopf algebra $\widehat{\mathbf{\Omega}}(C_*(X))$ by $\widehat{A}$. Let $i \colon A \to \widehat{A}$ be the natural inclusion map. As in section \ref{Hochschildadjoint}, we denote by $i_{ad}^*\widehat{A}$ to be $\widehat{A}$ equipped with the right dg $A$-module structure given by the adjoint action and $i^*(A)$ when equipped with the left $(A \otimes A^{op})$-action.

Note there is a natural isomorphism of dg $\mathbf{k}$-modules
\begin{equation} \label{equiv1}
C \otimes_{\iota} i^*_{ad} \widehat{A} \cong \mathcal{Q}(\mathbf{k}, C, A) \otimes_{A} i^*_{ad} \widehat{A},
\end{equation}
where $\mathbf{k}$ is thought of as a right dg $A$-module via the augmentation map.

By Proposition \ref{contraction}, there is a natural chain homotopy equivalence
\begin{equation} \label{equiv2}
\mathcal{Q}(\mathbf{k}, C, A) \otimes_{A} i^*_{ad} \widehat{A} \leftrightarrows \text{Bar}(\mathbf{k}, A, A) \otimes_{A} i^*_{ad} \widehat{A} \cong \mathcal{C}_*(A, i^*_{ad}\widehat{A}).
\end{equation}
By Proposition \ref{HochschildHopfalgebra}, there is a natural isomorphism of dg $\mathbf{k}$-modules
\begin{equation} \label{equiv3}
 \mathcal{C}_*(A, i^*_{ad}\widehat{A}) \cong \mathcal{CH}_*(A, i^*\widehat{A})= \text{Bar}(A,A,A) \otimes_{A \otimes A^{op}} i^*\widehat{A}
\end{equation}
Again by Proposition \ref{contraction}, there is a natural chain homotopy equivalence
\begin{equation} \label{equiv4}
\text{Bar}(A,A,A) \otimes_{A \otimes A^{op}} i^*\widehat{A} \leftrightarrows \mathcal{Q}(A,C,A)\otimes_{A \otimes A^{op}} i^*\widehat{A}.
\end{equation}
Now observe there is a natural isomorphism of dg $\mathbf{k}$-modules
\begin{equation} \label{equiv5}
\mathcal{Q}(A,C,A)\otimes_{A \otimes A^{op}} i^*\widehat{A} \cong \widehat{\mathcal{coCH}}_*(C).
\end{equation}
Putting together \ref{equiv1}, \ref{equiv2}, \ref{equiv3}, \ref{equiv4}, and \ref{equiv5} yields the desired result. 
\end{proof}
\begin{remark} Note that the $\mathbf{k}$-linear map $C_*(X) \to \widehat{\mathcal{coCH}}_*(C_*(X))$ given by sending $\sigma \mapsto \sigma \otimes 1_{\mathbf{k}}$ is \textit{not} in general a chain map (it would be if $C_*(X)$ was strictly cocommutative, which is not). However, we do have a chain map \[C_*(X) \to C_*(X) \otimes_{\iota} i^*_{ad} \widehat{\mathbf{\Omega}}(C_*(X))\]
given by the same formula \[\sigma \mapsto \sigma \otimes 1_{\mathbf{k}}.\]

Composing this map with the chain map 
\[C_*(X) \otimes_{\iota} i^*_{ad} \widehat{\mathbf{\Omega}}(C_*(X)) \to \widehat{\mathcal{coCH}}_*(C_*(X))\] given in Theorem \ref{browncohoch}, we obtain a chain map
\[C_*(X) \to \widehat{\mathcal{coCH}}_*(C_*(X))\]
modeling the continuous map $X \to LX$ that sends a point $b$ in $X$ to the constant loop at $b$. We expect this map to be useful in studying the Goresky-Hingston coproduct and the corresponding Lie cobracket in the $S^1$-equivariant setting in the string topology of non-simply connected manifolds. 
\end{remark}
\printbibliography

\end{document}